\documentclass[11pt]{article}
\usepackage[utf8]{inputenc}

\usepackage{amsmath} 
\usepackage{amssymb}
\usepackage{amsthm}
\usepackage{graphicx}
\usepackage{amscd}
\usepackage{url}
\usepackage{xcolor}
\usepackage{todonotes}
\usepackage{enumerate}
\usepackage{enumitem}
\usepackage{comment}
\usepackage{bbm}
\usepackage{hyperref}

\newtheorem{theorem}{\bf Theorem}[section]
\newtheorem{lemma}[theorem]{\bf Lemma}
\newtheorem{cor}[theorem]{\bf Corollary}
\newtheorem{proposition}[theorem]{\bf Proposition}
\newtheorem{problem}[theorem]{\bf Problem}
\newtheorem{prop}[theorem]{\bf Proposition}

\theoremstyle{definition}

\newtheorem{remark}[theorem]{\bf Remark}

\newcommand{\LC}{\mathrm{LC}}
\newcommand{\AG}{\mathrm{AG}}

\newcommand{\Z}{\mathbb{Z}}

\newcommand{\N}{\mathbb{N}}

\newcommand{\dd}{\frac{p-1}{2}}
\newcommand{\legendre}[2]{\ensuremath{\left( \frac{#1}{#2} \right) }}

\newcommand{\rr}{r_\gamma}
\newcommand{\su}{\sum_{l=1}^{\dd}}
\newcommand{\sua}{\sum_{l=1}^{\dd} \legendre{a_l-\gamma}{p}}
\newcommand{\sub}{\sum_{l=1}^{\dd} \legendre{b_l-\gamma}{p}}

\newcommand{\F}{\mathbb{F}}
\newcommand{\ff}{\mathbb{F}}

          \usepackage{array}

\title{On polynomials of small range sum}

\author{Gergely Kiss\thanks{Alfréd Rényi Institute of Mathematics and Corvinus University of Budapest, The author is supported by the J\'anos Bolyai Research Fellowship, ÚNKP-23-5 New National Excellence Program of the Ministry of Culture and
Innovation from the source of the National Research, Development and Innovation Fund, OTKA grants FK 142993, STARTING 150576.
E-mail: {\tt kigergo57@gmail.com}} 
\and Ádám Markó \thanks{E\"otv\"os Lor\'and University, Budapest, Hungary, supported by OTKA grant FK 142993.
E-mail: {\tt marqadam@gmail.com}} 
\and Zolt\'an L\'or\'ant Nagy\thanks{ELTE Linear Hypergraphs  Research Group,
 E\"otv\"os Lor\'and University, Budapest, Hungary. The author is supported by the Hungarian Research Grant (NKFIH) no. PD  134953 and no. 124950.  	E-mail: {\tt nagyzoli@cs.elte.hu}}
 \and Gábor Somlai \thanks{The University of Melbourne, Australia and E\"otv\"os Lor\'and University, Budapest, Hungary, supported by  OTKA Grants K138596, STARTING 150576, SNN	152582, and by ARC Discovery Project DP250104965.
E-mail: {gabor.somlai@unimelb.edu.au and gabor.somlai@ttk.elte.hu }}}

\date{}

\begin{document}

\maketitle

\begin{abstract}
In order to reprove an old result of Rédei on the number of directions determined by a set of cardinality $p$ in $\mathbb{F}_p^2$, Somlai proved that non-constant polynomials over the field $\mathbb{F}_p$ whose range sum is equal to $p$ are of degree at least $\frac{p-1}{2}$. Here the values are represented by integers in
 $[0,p-1]$.
In this paper, we characterise all of these polynomials having degree exactly $\frac{p-1}{2}$, if $p$ is large enough.
As a consequence, for the same set of primes, we re-establish the characterisation of sets with few determined directions due to Lovász and Schrijver using discrete Fourier analysis. 
\end{abstract}
\section{Introduction}
Given a set $S$ of $q$ points in the affine plane $\AG(2,q)$, where $q$ is a prime power, the classical direction problem,  first investigated by Rédei \cite{redei}, aims to determine the possible size of the direction set  of $S$, defined as $D=\{(a_i-a_j)/(b_i-b_j): 
 (a_i, b_i), (a_j, b_j)\in S, ~ i \ne j\}\subseteq \F_q \cup \{\infty\}$. 
Rédei \cite{redei} proved that if $S$ is a set of size $p$ in the affine plane $\AG(2,p)$, where $p$ is a prime, then $S$ is either a line or determines at least $\frac{p+1}{2}$ directions. This result was extended by Megyesi to exclude 
$\frac{p+1}{2}$, namely 
a set of size $p$ is either a line or determines at least 
$\frac{p+3}{2}$ directions.
The proofs rely on the usage of Rédei's polynomials, which have become a useful tool in finite geometry.
Later, the exact same result was obtained by Dress, Klin and Muzychuk \cite{DKM}, who exploited this fact to prove an old theorem of Burnside describing transitive permutation groups of degree $p$. 

Sets determining exactly $\frac{p+3}{2}$ directions exist and they were explicitly described by Lovász and Schrijver \cite{LS}. They proved that up to an affine transformation there is a unique set of this sort. 
Gács \cite{gacs} proved that there is another gap in the number of directions so a set of size $p$ in $\AG(2,p)$ is either a  line or the set determined by Lovász and Schrijver or determines at least $\lfloor 2\frac{p-1}{3} \rfloor+1$ directions. 

A new method to handle this set of problems arose in \cite{somlai}, where projection polynomials (defined in the next section) are used. 
The concept of projection polynomials appeared in a paper of Kiss and Somlai \cite{KS}.  
The idea of identifying the elements of $\Z_p$ with integers is also used in a similar context in a paper of De Beule, Demeyer, Mattheus, and Sziklai \cite{BDMS}, where 
lifting functions were introduced to treat polynomials over $\F_p$ as functions from $\F_p$ to $\{0,1, \ldots, p-1 \}\subset \N$.
This approach connects the direction problem to a natural general problem.
\begin{problem}
    Given a multiset $M$ of size $p$ over $\F_p$, determine the minimum degree of a polynomial over $\ff_p$  with range $M$.
\end{problem}

Recall that over any finite field $\ff_q$, any
function can be uniquely represented by a polynomial of degree at most $q-1$, as $x^q = x$ for all $x\in \F_q$. The degree of such a polynomial is called the \rm{reduced degree} of the polynomial (function). Thus for $q=p$ prime, the minimum (reduced) degree of a polynomial over $\ff_p$  with range $M$ is at most $p-1$.
It is easy to observe, on the other hand (see also Lemma \ref{lem:sumofvalues p-1coeff} in Section 2) that the reduced degree of a polynomial over $\ff_p$ must be $p-1$ unless the sum of the elements in the range is divisible by $p$. 
Thus, in terms of the sum of the elements, the first interesting question to study is when the sum of the elements, considered as elements from $[0,p-1]$, is equal to $p$. Throughout the paper, we use the term {\it range sum} for the sum of the admitted values of a polynomial with multiplicity, evaluated over all elements of a given prime field $\F_p$, where the admitted values are considered as elements from  $[0,p-1]\subset \mathbb{Z}$, i.e., $P ~\colon ~\mathbb{F}_p \to \{0,1, \ldots, p-1\}$. In order to emphasize this in our notation, we write $[P(x)]$ if we consider $P$ as an integer valued function. 
  It was proved in \cite{somlai} that if the range  sum  of a polynomial  over $\F_p$ is $p$, then either the polynomial is the constant $1$ or its degree is at least $\frac{p-1}{2}$.

Notice  that there are polynomials of degree $\frac{p-1}{2}$ whose range sum is $p$. 
The uniqueness of these polynomials is also asked in \cite{somlai}. We answer this question by proving some sort of uniqueness.
\begin{theorem}\label{thm:main}
Assume $p\ge520219910$ is a prime.
Up to affine transformations ($x \to ax+b$, where $a, b\in \F_p$ and $a\ne 0$), there are two polynomials $P$ of degree $\frac{p-1}{2}$ over $\F_p$ with $\sum_{x \in \F_p} [P(x)]=p$. These are the following. 
\begin{enumerate}[label=(\roman*)]
         \item\label{item:a} $x^{\frac{p-1}{2}}+1$,
    \item\label{item:b} $\frac{p+1}{2}\left( x^{\frac{p-1}{2}}+1 \right)$.
\end{enumerate}
For sake of completeness, the set of polynomials satisfying the conditions of the theorem are the following. 
\begin{enumerate}[label=(\roman*)]
     \item\label{itemaa}  $\pm (x-a)^{\frac{p-1}{2}}+1$,
    \item\label{itembb} $\frac{p+1}{2}\left( \pm  (x-a)^{\frac{p-1}{2}}+1 \right),$
    \end{enumerate}
 where $a \in \F_p$.
 \end{theorem}
We believe that the assumption on the prime $p$ is only a technicality, and the result stands without it.

For further results on the connection between the multiset $M$ and the degree of the corresponding polynomial, we refer to \cite{Biro} (for the case when $M$ consists of $0$s and $1$s), \cite{Gacs2}  (for the case when the degree is at least $p-2$) and \cite{Mura} (for the case when the degree is at least $p-t$ for some $t>p/2$). 

Next we highlight the connection between the degrees of polynomials and the number of determined directions.
The following proposition was proved in \cite{KS}. 
\begin{proposition}\label{prop:+2}
    If $ p-2 \ge d \ge 1$ and the projection polynomial of a set $S\subset \AG(2,p)$ is of degree $d$, then $S$ determines at least $d+2$ directions.
\end{proposition} The proof of this result uses Rédei's polynomial as a core tool. Since the investigation of this paper heavily relies on Proposition \ref{prop:+2}, we cannot avoid the implicit use of the original approach. As it is observed in \cite{somlai}, by combining the result on the sum of values in \cite{somlai} and the one on the degrees of the projection polynomials and the number of directions in \cite{KS}, we obtain a new proof for the Rédei-Megyesi result on the number of directions determined by a set of size $p$ in $\AG(2,p)$. 

Notice 
that the special case of $p$-element sets of the result in \cite{KS}, mentioned in the previous paragraph, can be derived from an earlier result of G\'acs \cite{Gacs_dir}. The projection polynomial in direction $m$ of a subset  $S=\{ (a_i,b_i) \colon i=0,1, \ldots, p-1\}$ of $\F_p^2$ can be written as follows: \[ f_m(t):=p-\sum_{i=0}^{p-1} (b_i-m a_i-t)^{p-1}. \] 
The coefficient of $t^j$ is equal to $(-1)^j \binom{p-1}{j}\sum_i (b_i-ma_i)^{p-1-j}$.

Putting $a_i=x$ and $b_i=f(x)$, we define 
 the polynomials $g_j$ ($1\le j\le p-1$)  as in \cite[Lemma 2.4]{Gacs_dir}, i.e., \[g_j(m):=\sum_{x\in \F_p}(f(x)+mx)^j
 .\] 
Clearly $\deg(g_j)\le  j-1$, since the leading coefficient cancels out in $\F_p$. 
Let $D$ denote the number of directions determined by the set $S$. 
If $-m$ is an undetermined direction of $S$, then $f(x)+mx$ is a bijection and hence the polynomials $g_j(m)=\sum_{u\in \F_p}u^j=0$. Thus if $1\le j\le p-D$, then $g_j$ has more roots than its degree, so $g_j$ is identically zero. Hence $\deg F_m\le (p-1)-(p-D+1)=D-2$, which implies Proposition \ref{prop:+2} for sets of size $p$.

Similar arguments were applied in \cite{GSzL} concerning the direction problem when the order of the field is $p^2$.

It is natural to hope that the characterisation of polynomials of Theorem \ref{thm:main} would also give us back the descriptive result of Lovász and Schrijver \cite{LS} for large enough primes, and indeed as a Fourier analytic application of Theorem \ref{thm:main} we obtain a new proof for the following theorem. 
\begin{theorem}\label{thm:lovaszschrijver}
Assume $p \ge 47$ is a prime which  satisfies the conclusion of Theorem \ref{thm:main}. Then up to an affine transformation, there is a unique set in $\AG(2,p)$ of cardinality $p$ determining exactly $\frac{p+3}{2}$ directions. All of these sets are affine equivalent to
$$\{(0, x): x^{(p-1)/2} = 1\}\cup\{(x, 0): x^{(p-1)/2} = 1\}\cup\{(0, 0)\}.$$ 
\end{theorem}
The previous presentation is borrowed from Szőnyi \cite{around} due to its simplicity. Let us remark that the above result was reproved and extended by Gács \cite{Gacs_dir}.
As an additional benefit we provide an explicit description of the Lovász-Schrijver construction based on projection polynomials listed in Theorem \ref{thm:main}. \\

{At the end of this section, we outline the proof of our results while presenting the organisation of the paper.
In Section 2, we summarize the basic facts concerning polynomials over finite fields and set the main notation of the paper. Next, we present the proof of an auxiliary result concerning the sum of Legendre symbols evaluated on translates of a set of density $1/2$ in Subsection 2.1. This is closely related to the parameters of the Paley graph, and might be of independent interest as well. The core of the proof of our main result, Theorem \ref{thm:main}, is carried out in Section 3.  
First we show that polynomials of degree $\frac {p-1}2$ with range sum $p$ are fully reducible and have distinct roots. Then we prove a list of identities involving the leading coefficient ($\LC$, for brevity) of the polynomial,  the roots and the elements of the range. Using an affine transformation, we may suppose that the $\LC$ is positive and is at most $\frac{p-1}2$. From the identities, we will deduce that the leading coefficient cannot take a value greater than $1$ and smaller than $\frac{p-1}2$. Here we distinguish between three cases depending on the value of the $\LC$, and use several tricks besides a deeper result which is related to the Hasse-Weil bound. At the end of this section, we handle the cases $\LC(P)=1$ and  $\LC(P)=\frac{p-1}2$ separately and show the uniqueness of the polynomial in both cases, up to affine transformation. 
Finally, for large enough primes we derive the Lovász-Schrijver theorem from our result in Section 4 and pose some open problems in Section 5.}

\section{Basic ideas and notation}
Let $p$ denote a prime number.
A multiset $M$ of size $p$ is called the range of a polynomial $P$ if $M=\{[P(x)]: x \in \F_p\}$.  In this case, the elements of the range of $P$  are identified with the corresponding representatives from $\{ 0,1, \ldots, p-1\} \subset \Z$.

For general finite fields $\F_q$, where $q$ is a prime power, any function over $\F_q$ can uniquely be represented by a polynomial  $P$ of degree at most $q-1$. 

\begin{lemma}\label{lem:sumofvalues p-1coeff} Suppose that $P(x)=\sum_{k=0}^{p-2} a_kx^k$ is a polynomial over $\F_p$. Then $$\sum_{x\in \F_p} x^tP(x)=-a_{(p-1)-t}$$ holds over $\F_p$, for any $t\in [1,p-1] \cap \Z .$
\end{lemma}
\begin{proof}
Observe that $\sum_{x\in \F_p} x^t=0$ if $1\leq t\leq p-2$ 
and $\sum_{x\in \F_p} x^{p-1}=-1$.
\end{proof}
\begin{cor}\label{core} Suppose that $P$ is a reduced degree polynomial over $\F_p$ and $t$ is an integer with $1 \le t \le p-1$. Then 
    the degree of $P$ is smaller than $p-t$ if and only if $\sum_{x\in \F_p}x^{d}P(x)=0$ holds over $\F_p$, for all $0 \le d \le t-1$.\\
To put it differently, for every non-negative integer $\gamma$ we have $$\sum_{x\in \F_p}x^{\gamma}P(x)\neq 0 \Rightarrow \deg P\geq p-1-\gamma.$$ and $$  \deg P \geq p-1-\gamma\Rightarrow \sum_{x\in\F_p}x^{\delta}P(x)\neq 0 \mbox{ \ for \ some \  } 0 \leq \delta\leq \gamma.$$
\end{cor}

For a pair of sets $A$ and $B$, $A+B$ denotes the Minkowski-sum of the sets, i.e., $x\in A+B$ if and only $x$ is of form $a+b$ for some $a\in A$ and $b\in B$. If $B=\{b\}$ is a singleton, we might omit the brackets for brevity.

Let $S$ be a set in $\F_p^2$ which does not contain a line, and let $m$ be an element of $\F_p \cup \{ \infty \}$. Then the function $f_m$ is defined as follows. Let $f_m(t)$ be the number of elements of $S$ on the line $y=mx+t$. If $m=\infty$, then we take the line as $x=t$. Now $f_m$ can be viewed as a function from $\F_p$ to $\{0,1, \ldots, p-1 \}$. 
Thus we obtain a polynomial which we call the $m$-th \textrm{projection polynomial}.

\subsection{Intersections of translates of the set of quadratic residues}

In this subsection we prove a 
combinatorial statement applied in the next section.
It states that if, generally speaking, a system of sets over a given ground set is given which has a relatively rigid intersection pattern, meaning that their $k$-wise intersections are bounded from above via their sizes, then for any large enough set over the ground set we can find a representative from the set system which intersects it in a small subset.
This idea will be applied to a system of translates of the quadratic residues over a finite field $\F_p$, to obtain a key statement, Proposition \ref{1/7}.

We begin with a lemma that is used several times in the next sections. The lemma relies on the knowledge of the number of solutions of any quadratic equation. 
In Lemma \ref{lem:paleyparameters}, we use the character notation for the Legendre symbol. This is done because, with appropriate modifications, this statement can obviously be extended to characters as well.
\begin{lemma}\label{lem:paleyparameters}
	Let $p$ be an odd prime, let $\chi=\left(\frac{\cdot}{p}\right)$ be the Legendre symbol
	extended by $\chi(0)=0$, and let
	$
	K \subseteq \F_p$ with  $|K|=\frac{p-1}{2}$.
	For $\gamma\in\F_p$, put
	\[
	S(\gamma):=\sum_{k\in K}\chi(k-\gamma).
	\]
Suppose that for some $\gamma_0\in\F_p$ and some $0<t \le \frac{p-3}{4}$ one has
	\[
	|S(\gamma_0)|=\frac{p+1}{4}+t.
	\]
	Then for every $\gamma\neq \gamma_0$,
	\[
	|S(\gamma)|\le \frac{p+1}{4}-t.
	\]
\end{lemma}

\begin{proof}
	We first prove the following pairwise estimate: for any distinct $\gamma,\gamma'\in\F_p$ and any
	$\varepsilon,\eta\in\{\pm1\}$,
	\begin{equation}\label{eq:master}
	\varepsilon S(\gamma)+\eta S(\gamma')\le \frac{p-\varepsilon\eta}{2}.
	\end{equation}
	
	Fix distinct $\gamma,\gamma'\in\F_p$, and write $d=\gamma'-\gamma\neq 0$.
	For $\varepsilon,\eta\in\{\pm1\}$, let
	\[
	N_{\varepsilon,\eta}:=
	\#\{x\in \F_p:\chi(x-\gamma)=\varepsilon,\ \chi(x-\gamma')=\eta\}.
	\]
	For each $k\in K$, the quantity
	\[
	\varepsilon\chi(k-\gamma)+\eta\chi(k-\gamma')
	\]
	is at most $2$, and it is equal to $2$ precisely when
	$
	\chi(k-\gamma)=\varepsilon \mbox{ and } \chi(k-\gamma')=\eta.
	$
	Hence
	\begin{equation}\label{eq:first-bound}
	\varepsilon S(\gamma)+\eta S(\gamma')
	\le
	2N_{\varepsilon,\eta}
	+\max\{\varepsilon\chi(d),0\}
	+\max\{\eta\chi(-d),0\}.
	\end{equation}
	Indeed, the only exceptional points are $k=\gamma'$ and $k=\gamma$, where one of the two
	Legendre symbols vanishes.
	
	We now compute $N_{\varepsilon,\eta}$. For $\sigma\in\{\pm1\}$ define
	\[
	\mathbbm{1}_\sigma(x):=\frac{1+\sigma\chi(x)-\mathbf{1}_{\{0\}}(x)}{2},
	\]
	where $\mathbf 1_{\{d\}}(x)=1$ if $x=d$ and $\mathbf 1_{\{d\}}(x)=0$ otherwise.
	Then $\mathbbm{1}_\sigma(x)$ is the indicator function of the set of $x\in\F_p$ such that $\chi(x)=\sigma$.
	Therefore
	\[
	N_{\varepsilon,\eta}
	=
	\sum_{x\in\F_p}\mathbbm{1}_\varepsilon(x)\,\mathbbm{1}_\eta(x-d).
	\]
	Expanding gives
	\[
	4N_{\varepsilon,\eta}
	=
	\sum_{x\in\F_p}
	\bigl(1+\varepsilon\chi(x)-\mathbf 1_{\{0\}}(x)\bigr)
	\bigl(1+\eta\chi(x-d)-\mathbf 1_{\{d\}}(x)\bigr).
	\]

	Using
	\[
	\sum_{x\in\F_p}\chi(x)=\sum_{x\in\F_p}\chi(x-d)=0,
	\qquad
	\sum_{x\in\F_p}\mathbf 1_{\{0\}}(x)=\sum_{x\in\F_p}\mathbf 1_{\{d\}}(x)=1,
	\]
	\[
	\sum_{x\in\F_p}\chi(x)\mathbf 1_{\{d\}}(x)=\chi(d),
	\qquad
	\sum_{x\in\F_p}\mathbf 1_{\{0\}}(x)\chi(x-d)=\chi(-d),
	\qquad
	\sum_{x\in\F_p}\mathbf 1_{\{0\}}(x)\mathbf 1_{\{d\}}(x)=0,
	\]
	we obtain
	\begin{equation}\label{eq:Paley-s}
		4N_{\varepsilon,\eta}
	=
	p-2-\varepsilon\chi(d)-\eta\chi(-d)
	+\varepsilon\eta\sum_{x\in\F_p}\chi(x)\chi(x-d). 
	\end{equation}
	We note the following classical result, which is also found in \cite{Lidl}
	\[
\sum_{x\in\F_p}\chi(x)\chi(x-d)=-1 \quad \mbox{if} \quad d \ne 0.
	\]

Consequently,
\begin{equation}\label{eq:Paley_vegso}
	4N_{\varepsilon,\eta}
	=
	p-2-\varepsilon\chi(d)-\eta\chi(-d)-\varepsilon\eta.
	\end{equation}
The value $N_{\epsilon,\eta}$ depends on $d$ and $\left(\frac{-1}{p} \right)$, so we collect the possible values of $N_{\epsilon,\eta}$ in  Table \ref{tab:tqr}. 
    \begin{table}[!ht]
    \centering
\[
\begin{matrix}
\begin{array}{|c|c|c|}
\hline
\big(\frac{\gamma'-\gamma}{p}\big)=1
& p\equiv 1 \pmod 4
& p\equiv 3 \pmod 4 \\[5pt]
\hline
N_{1,1}   & \dfrac{p-5}{4} & \dfrac{p-3}{4} \\[5pt]
\hline
N_{1,-1}  & \dfrac{p-1}{4} & \dfrac{p-3}{4} \\[5pt]
\hline
N_{-1,1}  & \dfrac{p-1}{4} & \dfrac{p+1}{4} \\[5pt]
\hline
N_{-1,-1} & \dfrac{p-1}{4} & \dfrac{p-3}{4} \\[5pt]
\hline
\end{array}
&
\begin{array}{|c|c|c|}
\hline
\big(\frac{\gamma'-\gamma}{p}\big)=-1
& p\equiv 1 \pmod 4
& p\equiv 3 \pmod 4 \\[5pt]
\hline
N_{1,1}   & \dfrac{p-1}{4} & \dfrac{p-3}{4} \\[5pt]
\hline
N_{1,-1}  & \dfrac{p-1}{4} & \dfrac{p+1}{4} \\[5pt]
\hline
N_{-1,1}  & \dfrac{p-1}{4} & \dfrac{p-3}{4} \\[5pt]
\hline
N_{-1,-1} & \dfrac{p-5}{4} & \dfrac{p-3}{4} \\[5pt]
\hline
\end{array}
\end{matrix}
\] 
    
    \caption{Intersection sizes of translates of quadratic (non)residues.}
    \label{tab:tqr}
\end{table}

\color{black}
Substituting this into \eqref{eq:first-bound}, we obtain
	\[
	\varepsilon S(\gamma)+\eta S(\gamma')
	\le
	\frac{p-2-\varepsilon\eta}{2}
	+\frac{-\varepsilon\chi(d)+2\max\{\varepsilon\chi(d),0\}}{2}
	+\frac{-\eta\chi(-d)+2\max\{\eta\chi(-d),0\}}{2}.
	\]
	Finally, for every $a\in\{\pm1\}$ one has
	$
	-a+2\max\{a,0\}=1.
	$
	Applying this twice yields
	\[
	\varepsilon S(\gamma)+\eta S(\gamma')
	\le
	\frac{p-2-\varepsilon\eta}{2}+\frac12+\frac12
	=
	\frac{p-\varepsilon\eta}{2},
	\]
	which proves \eqref{eq:master}.

Assume that
	\[
	\left|S(\gamma_0) \right|=\frac{p+1}{4}+t
	\qquad \left(0<t\le \frac{p-3}{4} \right),
	\]
	and let $\gamma\neq\gamma_0$.
	Choose $\varepsilon,\eta\in\{\pm1\}$ such that
	\[
	\varepsilon S(\gamma_0)=|S(\gamma_0)|,\qquad \eta S(\gamma)=|S(\gamma)|.
	\]
	Then \eqref{eq:master} gives
	\[
	|S(\gamma_0)|+|S(\gamma)|
	=
	\varepsilon S(\gamma_0)+\eta S(\gamma)
	\le
	\frac{p-\varepsilon\eta}{2}
	\le
	\frac{p+1}{2}.
	\]
	Therefore
	\[
	|S(\gamma)|
	\le
	\frac{p+1}{2}-|S(\gamma_0)|
	=
	\frac{p+1}{4}-t,
	\]
	as required.
\end{proof}

\color{black}
As in the previous lemma, let $\chi(x) = +1$ if $x\in \ff_q$ is a
nonzero square, $\chi(x)=-1$ if $x$ is a non-square and $\chi(0) =0$.
The lemma below is due to Szőnyi and is proved via the estimation of sums of multiplicative characters.

\begin{lemma}[Szőnyi, \cite{Szonyi_Weil}]\label{weil}
Let $q$ be an odd prime power and let $f_1, f_2, \ldots ,f_m\in \ff_q[x]$ be polynomials for which no partial product of form $\prod\limits_{i\in I, I\subseteq [1,m]} f_i(x)$ can be  written as a
constant multiple of a square of a polynomial.

Suppose that $ 2^{m-1} \sum_{i=1}^m \deg(f_i) < {\sqrt{q}- 1}$. Let $\textbf{v}\in \{+1, -1\}^m$, and  $N(\textbf{v})$ denote the number of elements $y\in \ff_q$ for which  $\chi(f_i(y))=v_i$  for every $i = 1, \ldots, m$. Then 
\begin{equation}\label{eq:szonyi}
    \left|N(\textbf{v})-\frac{q}{2^m} \right|< \sum_{i=1}^m \deg(f_i)\frac{\sqrt{q}+ 1}{2}.
\end{equation}  
\end{lemma}
We apply this result to linear polynomials of the form $f_i(x)=x+r_i$. Let $Q$ denote the nonzero square elements in $\ff_q$. Then Lemma \ref{weil} states that we can estimate the size of not only the intersection of translates $\bigcap_i (Q-r_i)$ but we can take the complement of each translate independently from each other.

Let $A_k \mbox{ } (k \in K)$ denote a collection of subsets of $\F_p$, where $K$ is a finite index set, and let $\bar{A_i}$ denote the complement of $A_i$. For a subset $I$ of $K$  let $E_I:=\left( \bigcap_{k \in I} A_k \right) \cap \left( \bigcap_{l \in K\setminus I} \overline{A_l} \right)$. 
\begin{lemma} Let $A_k$ be the set of nonzero square elements of $\F_p$ translated by pairwise distinct elements $r_k$ for $k\in K:=\{1, 2, \ldots, t\}$, where $t2^t<0.5\sqrt p-1$. Then 
     $$|E_I|=\left|\left( \bigcap_{k\in I} A_k \right) \cap \left( \bigcap_{l \in K \setminus I} \overline{A_l}\right) \right|<\frac{p}{2^t}+t\frac{\sqrt{p}+1}{2}+t.$$
\end{lemma}
\begin{proof}
    The statement follows from Lemma \ref{weil}. The extra factor $t$ comes from the fact that $\overline{A_l}$ contains $r_l$, while these elements are not counted in equation \eqref{eq:szonyi}.
  \end{proof}

Note that the former  statement is suggested as an exercise in the handbook of Lidl and Niederreiter \cite{Lidl},  see Ex. 5.64.
\begin{lemma}\label{halmazmetszetek}
Let $X$ be a finite set, and let $A_1,\dots,A_t,B\subseteq X$.
For each $I\subseteq [t]$, define
\[
E_I:=
\left(\bigcap_{i\in I}A_i\right)
\cap
\left(\bigcap_{j\in [t]\setminus I}(X\setminus A_j)\right).
\]
Then for every integer $1\le \hat r\le t$,
\[
\min_{i\in [t]} |A_i\cap B|
\le
\frac1t\left(
\sum_{r=\hat r}^t
\left|\bigcup_{|I|\ge r}E_I\right|
+
(\hat r-1)|B|
\right).
\]
\end{lemma}

\begin{proof}
The sets $E_I$, $I\subseteq [t]=\{1,2,\ldots ,t \}$, are pairwise disjoint. Indeed, if
$I\neq J$, choose $k\in I\triangle J$. Then one of $E_I,E_J$ is contained
in $A_k$, while the other is contained in $X\setminus A_k$.

Moreover, for each $k\in [t]$,
\[
A_k=\bigcup_{\substack{I\subseteq [t]\\ k\in I}}E_I.
\]
Consequently,
\[
A_k\cap B=
\bigcup_{\substack{I\subseteq [t]\\ k\in I}}(E_I\cap B),
\]
which is a union of pairwise disjoint sets. Hence
\[
|A_k\cap B|
=
\sum_{\substack{I\subseteq [t]\\ k\in I}} |E_I\cap B|.
\]
Summing over $k=1,\dots,t$, we get
\[
\sum_{k=1}^t |A_k\cap B|
=
\sum_{I\subseteq [t]} |E_I\cap B|\cdot |I|.
\]
Now use the identity
$
|I|=\sum_{r=1}^t \mathbf 1_{ \left\{ J \subset [t] \colon |J| \ge r \right\} }$, where $\mathbf{1}_H$ denotes the characteristic function of the set $H$ so $\mathbf{1}_H(x)=1$ if $x\in H$ and $\mathbf{1}_H(x)=0$ if $x\not\in H$. 
To obtain
\[
\sum_{I\subseteq [t]} |E_I\cap B|\cdot |I|
=
\sum_{r=1}^t
\sum_{\substack{I\subseteq [t]\\ |I|\ge r}}
|E_I\cap B|.
\]
Since the sets $E_I$ are pairwise disjoint, for each fixed $r$ we have
\[
\sum_{\substack{I\subseteq [t]\\ |I|\ge r}}
|E_I\cap B|
=
\left|
\left(\bigcup_{|I|\ge r}E_I\right)\cap B
\right|.
\]
Put
\[
H_r:=\bigcup_{|I|\ge r}E_I.
\]
Thus
\[
\sum_{k=1}^t |A_k\cap B|
=
\sum_{r=1}^t |H_r\cap B|.
\]

We now split the sum at $\hat r$. For $1\le r<\hat r$, we use
\[
|H_r\cap B|\le |B|,
\]
and for $\hat r\le r\le t$, we use
\[
|H_r\cap B|\le |H_r|.
\]
Hence
\[
\sum_{k=1}^t |A_k\cap B|
\le
(\hat r-1)|B|+\sum_{r=\hat r}^t |H_r|.
\]
Substituting the definition of $H_r$, this becomes
\[
\sum_{k=1}^t |A_k\cap B|
\le
(\hat r-1)|B|
+
\sum_{r=\hat r}^t
\left|\bigcup_{|I|\ge r}E_I\right|.
\]
Finally,
\[
\min_{i\in[t]} |A_i\cap B|
\le
\frac1t\sum_{k=1}^t |A_k\cap B|,
\]
so
\[
\min_{i\in [t]} |A_i\cap B|
\le
\frac1t\left(
\sum_{r=\hat r}^t
\left|\bigcup_{|I|\ge r}E_I\right|
+
(\hat r-1)|B|
\right).
\]
This proves the lemma.
\end{proof}

\begin{cor}\label{Cory0}
Let $p$ be an odd prime, and put
\[
Q_0:=\left\{x\in\F_p:\left(\frac{x}{p}\right)\in\{0,1\}\right\},
\qquad
N_0:=\left\{x\in\F_p:\left(\frac{x}{p}\right)\in\{0,-1\}\right\}.
\]
Let $A_1,\dots,A_8$ be eight pairwise distinct translates of $Q_0$, or eight
pairwise distinct translates of $N_0$. Let $B\subseteq\F_p$ with
\[
|B|=\frac{p-1}{2}.
\]
Assume that $p$ is large enough so that Lemma~2.4 applies to eight distinct
linear polynomials; for instance, $p>1025^2$ is sufficient. Then
\[
\min_{i\in[8]} |A_i\cap B|
<
\frac18\left(
140\left(\frac{p}{2^8}+8\frac{\sqrt p+1}{2}\right)
+32+2(p-1)
\right).f
\]
\end{cor}
\begin{proof} We prove the statement for translates of $Q_0$; the proof for translates of $N_0$ is identical. 
Write \[ A_i=Q_0+\rho_i, \qquad i=1,\dots,8, \] where the $\rho_i$'s are pairwise distinct elements of $\F_p$. For $I\subseteq[8]$, let \[ E_I:= \left(\bigcap_{i\in I}A_i\right) \cap \left(\bigcap_{j\in[8]\setminus I}(\F_p\setminus A_j)\right). \] 
For $r=1,\dots,8$, put \[ H_r:=\bigcup_{|I|\ge r}E_I. \] By Lemma~\ref{halmazmetszetek}, applied with $t=8$ and $\hat r=5$, we have \[ \min_{i\in[8]} |A_i\cap B| \le \frac18\left( \sum_{r=5}^8 |H_r|+4|B| \right). \] Since $|B|=(p-1)/2$, it remains to estimate $\sum_{r=5}^8 |H_r|$. Let \[ Z:=\{\rho_1,\dots,\rho_8\}. \] For $x\notin Z$, membership in $A_i=Q_0+\rho_i$ is equivalent to \[ \left(\frac{x-\rho_i}{p}\right)=1. \] 
Thus, outside $Z$, the set $H_r$ is covered by those sign patterns of the eight Legendre symbols \[ \left(\frac{x-\rho_1}{p}\right),\dots, \left(\frac{x-\rho_8}{p}\right) \] that have at least $r$ entries equal to $1$. By Lemma~2.4, each such sign pattern has fewer than \[ \frac{p}{2^8}+8\frac{\sqrt p+1}{2} \] solutions. The exceptional set $Z$ contributes at most $8$ points to each $H_r$. Hence \[ |H_r| < \left(\sum_{s=r}^8\binom{8}{s}\right) \left(\frac{p}{2^8}+8\frac{\sqrt p+1}{2}\right) +8. \] Summing this for $r=5,6,7,8$, and using \[ \sum_{r=5}^8\sum_{s=r}^8\binom{8}{s} = \sum_{i=0}^3\binom{8}{i} +\sum_{i=0}^2\binom{8}{i} +\sum_{i=0}^1\binom{8}{i} +1 = 140, \] we get \[ \sum_{r=5}^8 |H_r| < 140\left(\frac{p}{2^8}+8\frac{\sqrt p+1}{2}\right)+32. \] Substituting this into the estimate from Lemma~\ref{halmazmetszetek} gives the claim. \end{proof}

\begin{proposition}\label{1/7} Let \[ B=\{\alpha_1,\dots,\alpha_{(p-1)/2}\}\subseteq\F_p, \] where $p\ge 520219910$ is a prime. Then \[ \#\left\{ \gamma\in\F_p ~\colon ~
\left| \sum_{\ell=1}^{(p-1)/2} \left(\frac{\alpha_\ell-\gamma}{p}\right) \right| \ge \frac p7 \right\} \le 16. \] \end{proposition} 
\begin{proof} 
Put \[ S_B(\gamma):= \sum_{\alpha\in B}\left(\frac{\alpha-\gamma}{p}\right), \qquad \gamma\in\F_p. \] 
We first prove that \[ \#\left\{\gamma\in\F_p:S_B(\gamma)\ge \frac p7\right\}\le 7. \] As before, let \[ Q:=\left\{x\in\F_p^\times:\left(\frac{x}{p}\right)=1\right\}, \qquad Q_0:=Q\cup\{0\}. \] 
For every $\gamma\in\F_p$, we have \begin{equation}\label{eq:SBQ00} S_B(\gamma) = 2|(Q_0+\gamma)\cap B|-\frac{p-1}{2}-\mathbf 1_B(\gamma). \end{equation} 
Indeed, \[ \mathbf 1_{Q_0+\gamma}(\alpha) = \frac{1+\left(\frac{\alpha-\gamma}{p}\right)+\mathbf 1_{\{\gamma\}}(\alpha)}{2}, \] 
and summing over $\alpha\in B$ gives \eqref{eq:SBQ00}. 
Thus, if $S_B(\gamma)\ge p/7$, then \[ |(Q_0+\gamma)\cap B| = \frac12\left(\frac{p-1}{2}+S_B(\gamma)+\mathbf 1_B(\gamma)\right) \ge \frac{9p-7}{28}. \] 
Assume, for contradiction, that there are eight distinct elements \[ \gamma_1,\dots,\gamma_8\in\F_p \] such that \[ S_B(\gamma_i)\ge \frac p7 \qquad (i=1,\dots,8). \] 
Then \[ |(Q_0+\gamma_i)\cap B|\ge \frac{9p-7}{28} \qquad (i=1,\dots,8). \] 
Applying Corollary~\ref{Cory0} to the eight distinct translates \[ Q_0+\gamma_1,\dots,Q_0+\gamma_8, \] we get some $i\in[8]$ such that \[ |(Q_0+\gamma_i)\cap B| < \frac18\left( 140\left(\frac{p}{2^8}+8\frac{\sqrt p+1}{2}\right) +32+2(p-1) \right). \] 
For $p\ge 520219910$, the right-hand side is smaller than \[ \frac{9p-7}{28}, \] a contradiction. 
Therefore \[ \#\left\{\gamma\in\F_p:S_B(\gamma)\ge \frac p7\right\}\le 7. \] The negative tail is analogous. Let \[ N:=\left\{x\in\F_p^\times:\left(\frac{x}{p}\right)=-1\right\}, \qquad N_0:=N\cup\{0\}. \] 
For every $\gamma\in\F_p$, we have \begin{equation}\label{eq:SBN0} -S_B(\gamma) = 2|(N_0+\gamma)\cap B|-\frac{p-1}{2}-\mathbf 1_B(\gamma), \end{equation} because \[ \mathbf 1_{N_0+\gamma}(\alpha) = \frac{1-\left(\frac{\alpha-\gamma}{p}\right)+\mathbf 1_{\{\gamma\}}(\alpha)}{2}. \] 
Hence, if $S_B(\gamma)\le -p/7$, then \[ |(N_0+\gamma)\cap B|\ge \frac{9p-7}{28}. \] 
Repeating the previous argument with translates of $N_0$, we obtain \[ \#\left\{\gamma\in\F_p:S_B(\gamma)\le -\frac p7\right\}\le 7. \] 
Combining the two estimates gives the stronger bound \[ \#\left\{\gamma\in\F_p:|S_B(\gamma)|\ge \frac p7\right\}\le 14. \] In particular, the claimed bound $16$ follows. \end{proof} 
\begin{remark} The numerical assumption $p\ge 520219910$ is sufficient because \[ \frac{9p-7}{28} > \frac18\left( 140\left(\frac{p}{2^8}+8\frac{\sqrt p+1}{2}\right) +32+2(p-1) \right) \] is equivalent to \[ \frac{11}{3584}p-70\sqrt p-74>0, \] which holds for every $p\ge 520219910$. \end{remark}

\section{Proof of the main result}
Let us assume that $P$ is a polynomial whose range sum is $p$.  Recall that range sum refers to the sum of the admitted values of the polynomial $P$ with multiplicity, evaluated over all elements of a given prime field $\F_p$, where the admitted values are considered as elements from  $\{0, 1, \dots , p-1\}$. In the sequel we write $\sum_{x \in \F_p } [P(x)]=p$, where $[P(x)]$ is an integer between 0 and $p-1$.
We have already seen that $P$ is the constant 1 polynomial or its degree is at least $\dd$. The main goal of this section is to characterise those polynomials $P$ whose range sum is equal to $p$ and which are also of degree exactly $\dd$. As a consequence of this section, we establish the description of the polynomials given in Theorem \ref{thm:main}.

\subsection{Roots of a polynomial of degree \texorpdfstring{$\dd$}{}} 
The main purpose of this section is to prove that $P$ is fully reducible.
\begin{prop}\label{prop:fullyred}
Let $P$ be a polynomial of degree $\dd$ with $\sum_{x \in \F_p } [P(x)]=p$. 
Then $P$ is fully reducible.  Moreover, $P$ has $\dd$ different roots.       
\end{prop}
For a function $P$, we denote by $Supp(P)=\{x\in \F_p \mid [P(x)] \neq 0 \}$ the support of $P$.
In order to prove Proposition \ref{prop:fullyred} we prove the following. 
\begin{prop}\label{prop:22} Suppose that  $P$ is a nonconstant polynomial satisfying $\sum_{x\in \F_p} [P(x)]=p$ and $0<|\{y\in \F_p: P(y)=0 \}|\leq i<p-1$. Then $\deg P \geq p-1-i$. In other words, the conditions imply $\deg P \ge |Supp(P)|-1$. 
\end{prop}
\begin{proof}
Without loss of generality we may assume $P(0)=0$.
    Let $\{0 = a_1,a_2, \ldots, a_j\}$ ($j \le i$) denote the set of the roots of $P$, so these are different elements of $\F_p$. Since the sum of the values of $P$ is $p$ we can see that if $j\ge 1$, then there is $x \in \mathbb{F}_p$ such that $[P(x)]>1$. Now we build a multiset $S$ of the elements of
    $\mathbb{F}_p$. The multiplicity of $x$ in the multiset $S$ 
    is $[P(x)]-1$ if $[P(x)] \ge 1$ and $0$ otherwise. 
  The total weight of $S$ is clearly $j$.

Let us calculate $\sum_{x \in \mathbb{F}_p}[x^k P(x)] \in \Z$, where $k \le j$. We will compare this sum with 
$\sum_{x \in \mathbb{F}_p}x^k $, which is congruent to zero except if $k=p-1$. 
\begin{equation}\label{eq:22}
    \sum_{x \in \mathbb{F}_p}x^k P(x) \equiv \sum_{x \in \mathbb{F}_p}x^k - \sum_{l=1}^j a_l^k + \sum_{b \in S} b^k \pmod{p},
\end{equation}
where the numbers appearing in $S$ are counted with multiplicity
in the last sum. Instead we list the elements of the multiset $S$ as follows: $b_1,b_2, \ldots , b_j$, where $b_l$'s are not necessarily different. 

Assume now that $P$ is of degree less than $p-1-j$. Then we have $\sum_{x \in \mathbb{F}_p}x^k P(x) \equiv 0 \pmod{p}$ for $k \le j$. By equation \eqref{eq:22} we have $\sum_{l=1}^j a_l^k \equiv \sum_{l=1}^j b_l^k \pmod{p}$ if $k \le j $ since $\sum_{x \in \F_p} x^k \equiv 0 \pmod{p}$ for $k \le j < p-1$. 

Then by Newton's identities we have that the elementary symmetric polynomials of both $a_l$'s and $b_l$'s are the same. Since $a_1=0$ we have $\sigma_j(a_1, \ldots, a_j)=0$. This implies that $\sigma_j(b_1, \ldots, b_j)=0$. Then $b_{l'}=0$ for some $1 \le l' \le j$. This leads to a contradiction since $P(b_{l'})>1$ while $P(0)=0$. Notice that the fact that the symmetric polynomials of $a_i$ and $b_i$ are the same implies that these numbers are the same. 
 \end{proof}
 As it will be seen from the argument presented below, the proof of Proposition \ref{prop:22} essentially becomes trivial when $j \ge \frac{p-1}{2}$.
As a corollary of Proposition \ref{prop:22} we obtain a proof for Proposition \ref{prop:fullyred}.
\subsection{Identities for roots and values with multiplicities}
Let us denote by $c$ the leading coefficient of $P$. Let $a_1, a_2, \ldots , a_{\dd}$ denote the roots of $P$ and as in the proof of Proposition \ref{prop:22} let $\{ b_1,b_2, \ldots , b_k \}$ be a multiset, where the weight of $x \in \{0,1,\ldots, p-1\}$ is $[P(x)]-1$ if $[P(x)] \ge 1$ and $0$ otherwise. Since $\sum_{x \in \F_p}[P(x)]=p$ we have $k=\dd$.

It follows from equation \eqref{eq:22} that 
\begin{equation}\label{eq:a_lb_l}
    \sum_{l=1}^{\dd} a_l^{\dd} \equiv \sum_{l=1}^{\dd} b_l^{\dd} +c \pmod{p}
\end{equation}
It is well-known that $\left( \frac{a}{ p} \right) \equiv  a^{\dd}$ for every $a \in \F_p$. Thus we obtain.
\begin{equation}\label{eq:a_lb_llegendre}
    \sum_{l=1}^{\dd} \left( \frac{a_l}{ p} \right) \equiv \sum_{l=1}^{\dd} \legendre{b_l}{p} +c \pmod{p}
\end{equation}
It is clear that $-\dd \le \sum_{l=1}^{\dd} \legendre{x_l}{p} \le \dd$ for every $(x_l)_{l=1}^{\dd} \in \F_p^{\dd}$.
Thus, we obtain the following equality of integers 
\begin{equation}\label{eq:a_lb_llegendreres}
    \sum_{l=1}^{\dd} \left( \frac{a_l}{ p} \right) = \sum_{l=1}^{\dd} \legendre{b_l}{p} +r ,
\end{equation}
where $r=[c]$ or $r=[c]-p$. 

Using a suitable change of variables ($x \to x - \gamma$, $r\to r_{\gamma}$) of the polynomial $P$ we obtain 
\begin{equation}\label{eq:a_lb_l eltolva}
S_{\gamma}:=\sum_{l=1}^{\dd} \legendre{a_l-\gamma}{p} = \sum_{l=1}^{\dd} \legendre{b_l-\gamma}{p} +\rr. 
\end{equation}
The following is just a technical statement.
\begin{lemma}\label{lem:ds}
$ \sum_{\gamma \in \F_p}   \sum_{l=1}^{k} \legendre{x_l-\gamma}{p} =0$ for every $(x_l)_{l=1}^k \in \F_p^k$.
\end{lemma}
\begin{proof}
Since half of the nonzero elements of $\F_p$ are quadratic residues we have 
$ \sum_{\delta \in \F_p} \legendre{\delta}{p}=0$. Then changing the order of the summations gives the result. 
\end{proof}
\begin{cor}\label{cor:number of c and c-p} 
Let $r_{\gamma}$ denote the number defined in equation \eqref{eq:a_lb_l eltolva}. Then
\begin{itemize}
    \item $m_{c-p}:=\left| \left\{ \gamma \in \F_p  \colon \rr=[c]-p \right\} \right|=[c]$,
    \item $m_{c}:=\left| \left\{ \gamma \in \F_p  \colon \rr=[c] \right\} \right|=p-[c]$.
\end{itemize}
\end{cor}
\begin{proof}
    Lemma \ref{lem:ds} gives that $m_c [c]+ m_{c-p}([c]-p)=0$. The result follows from $m_c+m_{c-p}=p$.
\end{proof}
\subsection{Constraints on the leading coefficient}
The purpose of this section is to prove Theorem \ref{thm:leadingcoeff}. 

Let $k$ denote an absolute constant that satisfies
\begin{equation}\label{eq:1/7}
  \# \left\{ \gamma \in \mathbb{F}_p \colon \bigg \lvert  \sum_{l= 1}^{\frac{p-1}{2}}\legendre{a_l - \gamma}{p} \bigg \lvert \geq \frac{p}{7} \right\} \leq k.
\end{equation} 

{The framework of
Proposition \ref{1/7} guarantees that for any fixed integer $k \ge 16$, inequality \eqref{eq:1/7} holds if $p>p_0(k)$ for some $p_0\in \N$ depending only on $k$.
 With the explicit estimates in Section~2 before,
one may take $p_0=520219910$. 
From now on we assume $k=16$.
\begin{theorem}\label{thm:leadingcoeff}
    Let $p>p_0$ and $P$ be a polynomial of degree $\dd$ with $\sum_{x \in \F_p} [P(x)]=p$. Then $\LC(P)=1,p-1, \dd$ or $\frac{p+1}{2}$.
\end{theorem}
We prove this theorem using a case by case analysis. 
Let us first observe that if $\legendre{a}{p}=-1$, then $[\LC(P(ax))]=p-[\LC(P(x))]$ 
since the degree of $P$ is equal to $\dd$. Thus we may assume $1 \le [\LC(P)] \le \dd$. 

We recall that the leading coefficient of $P$ is denoted by $c$.
\begin{lemma}\label{lem:largeleadingcoeff}
Let $\frac{p-1}{2} > [c] > \frac{p-1}{2} - \frac{p}{7}=\frac{5p-7}{14}$.
    If $p>p_0$, then  
    $\LC(P) \ne c$.
\end{lemma}
\begin{proof}
 Let $\Gamma^{+}$ denote the elements $\gamma$ of $\mathbb{F}_p$ such that 
\begin{equation*}
    \sum_{l=1}^{\dd} \bigg( \frac{a_l- \gamma}p  \bigg) > -\frac{p}{7} ~~\mbox{ and } ~ \rr=[c]-p.
\end{equation*}

It follows from Corollary \ref{cor:number of c and c-p}, Proposition \ref{1/7} and $[c]>\frac{5p-7}{14}$
that
\begin{equation}\label{eq:gamma+becslese}
   \frac{5p-231}{14}< [c]- k \leq \lvert \Gamma^+ \lvert \leq c.
\end{equation}
For every $\gamma \in \Gamma^{+}$ we have that $\su \legendre{a_l-\gamma}{p}=\su \legendre{b_l-\gamma}{p}+([c]-p) $. Using $$\su \legendre{a_l-\gamma}{p} \ge - \frac{p}{7}$$ we obtain
\begin{equation}\label{eq:5/14}
     \su \bigg( \frac{b_l -\gamma }p \bigg) \geq \frac{p+1}{2} - \frac{p}{7} > \frac{5}{14}p.
\end{equation}
Thus
\begin{equation}\label{eq:lower bound ds}
  \su \sum_{\gamma \in \Gamma^+}  \legendre{b_l - \gamma}{p} = \sum_{\gamma \in \Gamma^+} \su \bigg( \frac{b_l-\gamma}p \bigg) \geq \frac{5}{14} \lvert \Gamma^+ \lvert p > \frac{5}{14}\frac{5p-231}{14}p.
\end{equation}
Notice that $\frac{5}{14}\frac{5p-231}{14}p>\frac{p-1}{2}\frac{p}{4}$ if $p>p_0$.
Let $b' \in \mathbb{F}_p$ for which $\sum_{\gamma \in \Gamma^+}  \legendre{b'- \gamma}{p}$ 
is maximal. 
Let $t \in \mathbb{Q}$  such that
\begin{equation}\label{eq:100}
     \sum_{\gamma \in \Gamma^+}  \bigg( \frac{b'-\gamma}{p} \bigg) = \frac{p}{4} + t.
    \end{equation}
It follows from $|\Gamma_+|\le \frac{p-1}{2}$ that $t\le \frac{p}{4}$. Equation \eqref{eq:lower bound ds} together with $\frac{5}{14}\frac{5p-231}{14}p>\frac{p-1}{2}\frac{p}{4}$
shows us that $t>0$.

We claim that for every $b \neq b'$, we have
\begin{equation}\label{eq:upper bound with t}
         \sum_{\gamma \in \Gamma^+}  \bigg(\frac{b - \gamma}p \bigg) < \frac{p}{4} - t+3 .
         \end{equation}
Indeed, let $g:=|\Gamma^+|$ and let $g_+$  and $g_-$ denote the number of elements $\gamma$ of $\Gamma^+$ such that $\legendre{b'-\gamma}{p}=1$ and $\legendre{b'-\gamma}{p}=-1$, respectively.  
Plainly $g_+ + g_-$ is equal to $g$ if $b' \not\in \Gamma_+$ and $g_+ + g_-=g-1$ otherwise.
 Further $g_+ - g_-=\frac{p}{4}+t$, which immediately gives $g_+=\frac{g+\frac{p}{4} +t} {2} \mbox{ or } \frac{g+\frac{p}{4} +t-1} {2}$.

From Table~\ref{tab:tqr} we see that, for fixed distinct $b,b'\in\F_p$,
the number of elements $\gamma\in\F_p$ satisfying
\[
\bigg(\frac{b'-\gamma}p \bigg)=1\qquad \mbox{and } \qquad \bigg(\frac{b-\gamma}p \bigg)=1
\]
is at most $(p+1)/4$. 
Thus, among the $g_+$ elements of $\Gamma^+$ for which
$\left(\frac{b'-\gamma}p \right)=1$, fewer than $p/4+1$ can also satisfy
$\left(\frac{b-\gamma}p \right)=1$, provided that $b\ne b'$.
Now 
\begin{equation*}
    \begin{split}
         \sum_{\gamma \in \Gamma_+}\left(\frac{b-\gamma}p \right)+\sum_{\gamma \in \Gamma_+}\left(\frac{b'-\gamma}p \right)\le 2\left|\left\{  \gamma \in \Gamma_+\bigg| \left(\frac{b-\gamma}p \right)=\left( \frac{b'-\gamma}p \right)=1 \right\}\right|+\mathbf 1_{\Gamma_+}(b)+\mathbf 1_{\Gamma_+}(b') \\ \le \frac{p+1}{2}+2.
             \end{split}
 \end{equation*}
Thus equation \eqref{eq:100} gives 
$$\sum_{\gamma \in \Gamma_+}\left(\frac{b-\gamma}p \right) \le \left( \frac{p+1}{2}+2 \right) - \left( \frac{p}{4}+t\right) <\frac{p}{4}-t+3.
$$
Now equation \eqref{eq:100} and \eqref{eq:upper bound with t} imply
\begin{equation}\label{eq:upper bound ds}
     \su  \sum_{\gamma \in \Gamma^+}  \bigg( \frac{b_l - \gamma}p \bigg) < \#\{b_l = b'\}\bigg(\frac{p}{4} + t\bigg) + \#\{b_l \neq b' \} \bigg(\frac{p}{4}-t + 3 \bigg).
\end{equation}
It follows from our assumptions (cf., \eqref{eq:gamma+becslese}) that
 $\lvert \Gamma^+ \lvert \geq \frac{5p-231}{14}$, which implies
\begin{equation*}
\frac{5p-231}{14} \frac{5}{14} p \leq     \frac{5}{14}\lvert \Gamma^+ \lvert p . 
\end{equation*}    
Now for $p >p_0$\footnote{In this case $p>2849$ would also be enough} and $k=16$ 
\begin{equation*}
         \frac{p-1}{2} \frac{p}{4} +  \frac{3}{2}p <\frac{p-1}{2} \frac{25}{98} p - \frac{p-1}{2} \frac{5k}{7}.
\end{equation*}  
Using equation \eqref{eq:upper bound ds} and \eqref{eq:lower bound ds} we obtain 
\begin{equation*}
\frac{p-1}{2}  \frac{p}{4}+ \frac{3}{2}p<
    \#\{b_l = b'\}\bigg(\frac{p}{4} + t\bigg) + \#\{b_l \neq b'\} \bigg(\frac{p}{4}-t + 3  \bigg) , 
\end{equation*}
so using $\#\{b_l \neq b'\}<p/2$, we have
\begin{equation*}
 \frac{p-1}{2}  \frac{p}{4}<   \#\{b_l = b'\}\bigg(\frac{p}{4} + t\bigg) + \#\{b_l \neq b'\} \bigg(\frac{p}{4}-t \bigg) ,
\end{equation*}
which gives us 
\begin{equation}\label{eq:b' dominans}
\frac{p-1}{4}    < \#\{ b_l = b' \}  . 
\end{equation}
We conclude that the role of $b'$ is dominant. More precisely, we have the following.
If $\legendre{b'-\gamma}{p} = 1 $, then more than $
\frac{p-1}{4}$ terms in 
$S:=\sum_{l=1}^{\frac{p-1}2} \legendre{b_l-\gamma}{p}$  are positive, implying that $\sum_{l=1}^{\frac{p-1}2} \legendre{b_l-\gamma}{p}>0$. 
Similarly, if $\legendre{b'-\gamma}{p} = -1 $, then $\sum_{l=1}^{\frac{p-1}2} \legendre{b_l -\gamma}{p} < 0 $. {Both cases occur for exactly $\frac{p-1}{4}$ many $\gamma \in \F_p$.} 
Finally, if  $b'= \gamma $, then more than $\frac{p-1}{4}$ terms in $S$ are zero and since every other term is either $1$ or $-1$ we have $\left| \sum_{l=1}^{\frac{p-1}2} \legendre{b_l -\gamma}{p}\right| < \frac{p+1}{4} $.  

\textbf{Case I}: Assume that 
\begin{equation}\label{eq:uj}
        \bigg \lvert \su  \legendre{a_l - \gamma}{p} \bigg \lvert \leq \frac{p}{4} +2.
\end{equation}
Since $p$ is large enough, we have  $ \bigg \lvert \su  \legendre{a_l - \gamma}{p} \bigg \rvert  < \frac{p-1}{2}-\frac{p}{7}$.
Then equation \eqref{eq:a_lb_l eltolva} implies that $\su   \legendre{b_l - \gamma}{p} > 0$ if and only if $\rr =[c]-p$, and $\su   \legendre{b_l - \gamma}{p} < 0$ if and only if $\rr =[c]$.
By Lemma \ref{lem:paleyparameters} we have that \eqref{eq:uj} holds for all but at most one $\gamma$.
Therefore, it follows that $\rr=[c]-p$ for at least $\frac{p-3 }{2}$ choices of $\gamma \in \F_p$. By Corollary \ref{cor:number of c and c-p} we have $[c] \geq \frac{p-3}2$.

It remains to handle the case when $c = \frac{p-3}{2}$.
We may strengthen our previous estimates using this assumption so we repeat the calculation.  

As before, $|\Gamma^+| \ge c-k$ so we have that $\lvert \Gamma^+ \lvert \geq \frac{p-3}{2}-k$.  The same argument as applied to obtain equation \eqref{eq:lower bound ds} gives
\begin{equation*}
    \sum_{\gamma \in \Gamma^+} \sum_{l=1}^\dd  \bigg( \frac{b_l-\gamma}{p} \bigg) \geq \frac{5}{14} \lvert \Gamma^+ \lvert p, 
\end{equation*}
and again
\begin{equation*}
    \# \{ b_l = b' \} \bigg( \frac{p}{4} + t \bigg)+ \# \{ b_l \neq b'\} \bigg( \frac{p}{4} - t + 3  \bigg) \geq \frac{5}{14} \frac{p(p-1)}{2} - (k+1)p.
\end{equation*}
Since $\#\{b_l\neq b'\}\leq \frac{p-1}{2}$, this implies 
$$ \# \{ b_l = b' \} \bigg( \frac{p}{4} + t \bigg) + \# \{ b_l \neq b'\} \bigg( \frac{p}{4} - t \bigg) \geq \frac{5}{14} \frac{p(p-1)}{2} - (k+1)p-\frac{3(p-1)}{2}. $$ For $p\ge 520219910$ and $k=16$, the right-hand side is larger than $$ \frac{2}{7}\frac{p(p-1)}{2}. $$
Simple rearrangement of the inequality gives 
\[
   ( \# \{ b_l = b'\} -  \#\{ b_l \neq b'\} )t \geq \frac{1}{28} \frac{p(p-1)}{2} .
\]
Since $t \leq \frac{p}{4}$ we have that
\[
     \# \{ b_l = b'\} -  \#\{ b_l \neq b'\}  \geq \frac{1}{7} \frac{p-1}{2}. 
\]

Using equation \eqref{eq:b' dominans}, this implies
\begin{equation}\label{eq:101}
    \# \{ b_l = b' \} \geq 
    \#\{ b_l \neq b'\} +
    \frac{1}{7} \frac{p-1}{2} >
    \frac{4}{7} \frac{p-1}{2},  
\end{equation}
which is an improvement for the lower bound on the multiplicity of $b'$ compared to equation \eqref{eq:b' dominans}.
For $\gamma = b'$ we get 
\[
   \bigg \lvert \su \legendre{b_l- b'}{p} \bigg \lvert \leq \frac{3}{7} \frac{p-1}{2}.  
\]
If 
$
\bigg \lvert \su  \legendre{a_l -b'}{p}  \bigg \lvert \leq \frac{p}{4} + 2   
$
then a simple triangle inequality gives
\[
    \bigg \lvert   \su  \legendre{a_l -b'}{p}  -  \su  \legendre{b_l - b'}{p}  \bigg \lvert \leq \frac{3}{7} \frac{p-1}{2} + \frac{p}{4} + 2 < [c]=\frac{p-3}{2}<\left|[c]-p\right|,
\]
a contradiction. 

\textbf{Case II.}:  
Assume that $$\bigg \lvert \su  \legendre{a_l -b'}{p}  \bigg \lvert > \frac{p}{4} + 2 . 
$$ 
By Lemma \ref{lem:paleyparameters}, there is at most one $\gamma' \in \mathbb{F}_p$ for which 
\begin{equation*}
\bigg \lvert \su \legendre{a_l -\gamma'}{p} \bigg \lvert > \frac{p}{4} + 2.   
\end{equation*}
Since this holds for $\gamma' = b'$, for all the other $\gamma$ (i.e. $\gamma \neq b'$) we have 
\begin{equation*}
\bigg \lvert \su \legendre{a_l -\gamma}{p}  \bigg \lvert \leq \frac{p}{4} +2.    
\end{equation*}
Since the multiplicity of $b'$ is large (see equation \eqref{eq:b' dominans}) the sign of  $\left( \frac{b'-\gamma}{p} \right) $ coincides with the one of $\sub$. Thus
\[
    \sua = \sub + (c-p),
\]
whenever $b'-\gamma$ is a quadratic residue, which happens exactly $\dd$ times. 
Corollary \ref{cor:number of c and c-p} then gives $c=\dd$, as required.
\end{proof}
The following lemma holds for $p$ large enough ($p>520219910$  according to Proposition \ref{1/7}).
\begin{lemma}\label{lem:37}
Let $p>520219910$.  Assume $ k < c' \leq \frac{p-1}{2} - \frac{p}{7} $. Then $\LC(P) \not\equiv c' \pmod{p}$. 
\end{lemma}
\begin{proof}
Assume indirectly that $LC(P)=c'$
Let us first look at the condition $c' > k$. In this case $r_{\gamma}=c'-p$ for more than $k$ different $\gamma \in \F_p$. It follows from Proposition \ref{1/7} that there exists $\gamma \in \F_p$ with 
\begin{equation*}
    \su \legendre{a_l - \gamma}{p} > -\frac{p}{7} \mbox{ and } \su \legendre{a_l - \gamma}{p}  = \su \legendre{b_l - \gamma}{p}  + (c'-p),
\end{equation*}
i.e. $\rr =c'-p$.
Clearly, 
$
\su \legendre{b_l -\gamma}{p}  \leq \frac{p-1}{2}$, which immediately implies
\begin{equation*}
 \su \legendre{a_l -\gamma}{p} =    \su \legendre{b_l -\gamma}{p}  + (c'-p) \leq -\frac{p}{7},
\end{equation*}
 a contradiction. 
\end{proof}
We will use the following in a way that it complements Lemma \ref{lem:37} so we will apply it if $p>8k$. 
\begin{lemma}\label{lem:kisc}
Assume $1 <c < \frac{1}{8}p$.  
Then $\LC(P) \ne c$.
\end{lemma}
\begin{proof}
  By Lemma \ref{lem:paleyparameters}, there is at most one $\gamma' \in \mathbb F _p$ such that 
\begin{equation*}
    \bigg \lvert \su \legendre{a_l - \gamma'}{p}  \bigg \lvert \geq \frac{p+2}{4}.
\end{equation*}
By Corollary \ref{cor:number of c and c-p} we have that $\rr= [c]-p$ for at least $2$ different $\gamma \in \F_p$ since $[c] \ge 2$.
 Thus there is a $\gamma \in \mathbb F_p$ ($\gamma\ne \gamma'$) such that
\begin{equation*}
     \su \legendre{a_l - \gamma}{p}= \su \legendre{b_l - \gamma}{p}+ ([c]-p)
~\mbox{ and }~ 
\bigg \lvert \sum_{l=1}^{\frac{p-1}{2}} \bigg(\frac{a_l-\gamma}p \bigg) \bigg \lvert \leq \frac{p+2}{4}.
\end{equation*}
Thus  \begin{equation*}
    p-[c]=\su \legendre{b_l - \gamma}{p}-\su \legendre{a_l - \gamma}{p} \le \left|\su \legendre{a_l - \gamma}{p}\right|+ \left| \su \legendre{b_l - \gamma}{p}\right| \le \frac{p-1}{2}+\frac{p+2}{4}=\frac{3p}{4}. \end{equation*}
It follows from our assumption $[c]<\frac{1}{8}p$ that $p-[c] > \frac{7}{8}p$, a contradicting $p-[c]\le 3p/4$.
\end{proof}

\subsection{Uniqueness}
In the previous subsections we proved that if $P \in \F_p[x]$ is a reduced polynomial whose range sum is equal to $p$ and $\deg(P)=\dd$, then $\LC(P) =\pm 1$ or $\LC(P) =\pm \dd$. It is clear that using a change of variable $x \to ax$, where $\legendre{a}{p}=-1$ we may assume that $\LC(P)=1$ or $\LC(P)=\dd$. Now we prove that for both $c=1$ and $c=\dd$ there is a polynomial $P$ of degree $\dd$ whose range sum is $p$ and $\LC(P)=c$.
\subsubsection{Uniqueness if \texorpdfstring{$a_{\frac{p-1}{2}}=\pm 1$}{}}
Let us assume that the sum of the values of $P$ is $p$ and $\deg(P)=\dd$ and $a_\dd= 1$. 
The polynomial $g_d:=P(x^2+d)$ has degree $p-1$ and leading coefficient $1$. By Lemma \ref{lem:sumofvalues p-1coeff} \[ \sum_{x\in\F_p}x^j=0 \quad (0\leq j\leq p-2), \qquad \sum_{x\in\F_p}x^{p-1}=-1, \] so we obtain \[ \sum_{x\in\F_p} g_d(x)\equiv -1 \pmod p \] for every $d\in\F_p$.

Let us now treat the values of the polynomial $g_d$ as integers in $\{0,1, \ldots, p-1 \}$.
Since 
$$S_d:=\sum_{x \in \F_p} [g_d(x)] =[P(d)]+ 2 \sum_{x \in d+(\F_p^*)^2   } [P(x)]$$
we have  $\sum_{x \in \F_p} [g_d(x)]\le 2p$. It follows that $\sum_{x \in \F_p} [g_d(x)]$
is either $p-1$ or $2p-1$. Now, if $[P(b)]$ is even, then $S_b=p-1$ and if $[P(b)]$ is odd, then $S_b=2p-1$. 

By Proposition \ref{prop:fullyred} $P$ has exactly $\dd$ roots. If $[P(x)]>1$ for every $x \in Supp(P)$, then the sum of the values of $P$ is larger than $p$ so we may assume $[P(b)]=1$ for some $b \in \F_p$. 
In this case 
\[
S_b=2p-1=1+2  \sum_{x \in b+(\F_p^*)^2   } [P(x)].
\]
This can only happen if $\sum_{x \in b+(\F_p^*)^2   } [P(x)]=p-1$. Since $\sum_{x \in \F_p}[P(x)]=p$, this implies that the quadratic nonresidues are roots of $P(x+b)$.
Therefore we have found all $\dd$ roots of a fully reducible polynomial with prescribed leading coefficients, proving the uniqueness of the polynomial. 

Note that the roots of $P(x+b)$ are the quadratic nonresidues so the corresponding polynomial is $x^{\frac{p-1}{2}}+1$. 
\subsubsection{Uniqueness if \texorpdfstring{$a_{\frac{p-1}{2}}=\pm \dd$}{}}
As in the previous case, we may assume $\LC(P)=\dd$.

\begin{proposition}\label{prop:unique c=p-1/2}
    If $LC(P)=c = \frac{p-1}{2}$, then
    \begin{equation*}
        P(x) = \frac{p-1}{2} (x-\beta')^{\frac{p-1}{2}} + \frac{p+1}{2} 
    \end{equation*}
    for some $\beta' \in \mathbb F_p$.
\end{proposition}

\begin{proof}
We have already seen that $P$ is fully
reducible. 
Thus we write 
$P(x) = 
\frac{p-1}{2} 
\prod_{k=1}^\dd(x
-a_l)$ and 
$\sum_{x \in 
\mathbb F_p} [P(x)] = p$. 
Further we 
have that  
$a_i \neq 
a_j$ if $i 
\neq j$. As it is defined above, let 
$\{b_1, b_2, 
\ldots, b_{\frac{p-1}{2}} \}$ be the multiset such that the multiplicity of $b_i$ is $max\left\{ P(b_i)-1,0 \right\}$.
The argument leading to equation \eqref{eq:101} in the proof of Lemma \ref{lem:largeleadingcoeff} uses only the lower bound \[ |\Gamma^+|\ge [c]-k. \] In the present case $[c]=\frac{p-1}{2}$, and Corollary \ref{cor:number of c and c-p} together with Proposition \ref{1/7} gives \[ |\Gamma^+|\ge \frac{p-1}{2}-k. \] Hence, the same argument gives equation \eqref{eq:101} in this case as well.
Thus, there exists $b' \in \mathbb{F}_p$ such that
\begin{equation*}
\big \lvert \{ 1\le l \le \dd \colon b_l=b' \} \big \lvert \geq \frac{4}{7} \frac{p-1}{2}. 
\end{equation*}
Notice first that $[c]=\dd$ and $[c]-p=-\frac{p+1}{2}$.
Further, if $b'-\gamma$ is a quadratic residue, then $\sub>0$ and then $r=-\frac{p+1}{2}=[c]-p$ while if $b'-\gamma$ is a quadratic nonresidue, then  $r=\dd=[c]$. By Corollary \ref{cor:number of c and c-p} we have that the multiplicity of $r=[c]-p$ is equal to $[c]=\dd$ so we have found all the occasions when this happens. Then in the remaining cases we need to have $r=[c]$. In particular this happens if $\gamma=b'$ so we have 
\begin{equation*}
    \su \bigg(\frac{a_l - b'}{p} \bigg) = \su \bigg(\frac{b_l- b'}{p} \bigg) + \frac{p-1}{2}.
\end{equation*}
By the special choice of $\gamma$
\begin{equation*}
   \bigg \lvert \su \bigg(\frac{b_l- b'}{p} \bigg) \bigg \lvert \leq \# \{b_l \neq b' \} \leq \frac{3}{7} \frac{p-1}{2},
\end{equation*}
we have that
\begin{equation}\label{eq:47}
    \su \bigg(\frac{a_l - b'}{p} \bigg) \geq  \frac{4}{7} \frac{p-1}{2}.  
\end{equation}
The set of $\{a_l \}_{l=1}^{\dd}$ and the multiset $\{b_l \}_{l=1}^{\dd}$ are disjoint, so $(\frac{a_l - b'}{p}) \neq 0$ for $1 \le l \le \dd$, so each summand in the previous equation is either 1 or $-1$ so equation \eqref{eq:47} becomes
\[
\bigg \lvert \bigg \{ 1 \le l \le \dd ~\bigg|~ \legendre{a_l - b'}{p}  = 1 \bigg \} \bigg \lvert -
\bigg \lvert \bigg \{ 1 \le l \le \dd ~\bigg|~ \legendre{a_l - b'}{p}  = -1 \bigg \} \bigg \lvert \ge \frac{4}{7}\dd.
\]
Thus we obtain 
\begin{equation}\label{eq:11/14}
    \bigg \lvert \bigg \{ 1 \le l \le \dd ~\bigg|~ \legendre{a_l - b'}{p}  = 1 \bigg \} \bigg \lvert \geq \frac{11}{14} \frac{p-1}{2}.
\end{equation}
    We now estimate the number of $x \in \mathbb{F}_p$ such that $P(x) \ne 1 $. 
\begin{equation*}
      \# \{ x\in \mathbb F_p \colon  [P(x)]> 1 \} \leq 1 + \# \{ x\in \mathbb F_p \colon [P(x)]> 1,~ x \neq b' \} \leq \frac{3}{7}\frac{p-1}{2} + 1. 
\end{equation*}
$P$ has exactly $\dd$ different roots so the support of $P$ is of cardinality $\frac{p+1}{2}$. Hence
\begin{equation}\label{eq:p(x)=1}
  \#  \{  x \in \mathbb{F}_p  \colon P(x) = 1 \}  \geq  \frac{p+1}{2}-\left(1+\frac{3}{7} \dd \right) = \frac{4}{7}  \frac{p-1}{2}. 
\end{equation}
Let $G: \mathbb F_p \to \mathbb F_p$ be defined as follows. 
\begin{equation*}
    G(x) := \frac{p-1}{2} (x- b' )^{\frac{p-1}{2}} + \frac{p+1}{2}  .
\end{equation*}
The range of $G$ consists of 3 elements and it is equal to $\left\{ 0,1,\frac{p+1}{2}\right\}$.
Plainly, $G(x) = 0 $ exactly when  $\left( \frac{x- b'}{p} \right) = 1$. Then we obtain from  equation \eqref{eq:11/14} that
\begin{equation*}
    \# \{ x \in \mathbb F_p \colon P(x) = G (x) = 0  \} \geq \frac{11}{14}\frac{p-1}{2}. 
\end{equation*}
Therefore
\begin{equation}\label{eq:G(x)=0}
        \# \left\{ x \in \mathbb F_p \colon G (x) = 0 \mbox{ and } P(x) \neq 0 \right\} \leq \frac{3}{14}\frac{p-1}{2}.
\end{equation}
Now by equation \eqref{eq:p(x)=1} and \eqref{eq:G(x)=0}

\begin{equation*}
   \begin{split}
 &        \{ x \in \mathbb F_p \colon P(x) = G (x) = 1 \} 
   = \# \left\{ x \in \mathbb F_p \colon \left(\frac{x-b'}{p} \right) = -1 \mbox{ and } P(x) = 1  \right\} 
\\
&\geq   \{ x \in \mathbb F_p \colon  P(x) = 1  \} - \frac{3}{14} \frac{p-1}{2}-1 \geq
 \frac{4}{7} \frac{p-1}{2} - \frac{3}{14} \frac{p-1}{2} -1= \frac{5}{14} \frac{p-1}{2} -1.
    \end{split}
   \end{equation*}
We have that 
\begin{equation*}
    \# \{ x \in \mathbb F_p \colon P(x) = G (x) \} \geq \frac{11}{14} \frac{p-1}{2} + \frac{5}{14} \frac{p-1}{2}-1 = \frac{16}{14} \frac{p-1}{2}-1.
\end{equation*}
This means that the polynomial $H(x) := P(x)- G(x) $  has at least $\frac{4}{7} {(p-1)}-1$ roots. Also the polynomials $P$ and $G$ have equal leading coefficients so $\deg H < \frac{p-1}{2}$. This is only possible if $H \equiv 0$.  
\end{proof}
\section{The Lov\'asz-Schrijver construction}
Lov\'asz and Schrijver showed \cite{LS} that up to affine transformations (elements of $\mathrm{AGL}(2,p)$) there is only one set in $\F_p^2$  determining exactly $\frac{p+3}{2}$ directions. We mention that   
the characterisation of this set also follows from Rédei’s characterisation of the solutions of his Problem II for $q = p$ \cite{redei}.
Such a set can be described due to Szőnyi \cite{around} as follows. 

Let $\{ (0,x) \colon x \in \F_p \}$ and $\{ (x,0) \colon x \in \F_p \}$ be identified with $\F_p$. Then we may also identify $\F_p^*$ with 
$\{ (0,x) \colon 0\ne x \in \F_p \}$ and $\{ (x,0) \colon 0 \ne  x \in \F_p \}$. 
Let  
$$L:=\left\{ (0,x) \colon \left(\frac x p \right)=1 \right\} \bigcup \left\{ (x,0) \colon \left(\frac x p \right)=1 \right\} \bigcup \left\{(0,0) \right\}.
$$
It is clear that $|L|=p$ and it determines exactly $\frac{p+3}{2}$ directions. 
Notice that the projection polynomials determined by $L$ are described in Theorem \ref{thm:main}. Up to affine transformations there are two polynomials listed in Theorem \ref{thm:main}. We may see that both types appear as projection polynomials of $L$. Moreover, the ones listed in  
Theorem \ref{thm:main} \ref{item:b} appears exactly twice.
Since $L$ is equidistributed in the set of parallel lines not determining a direction of $L$ we may see that there are exactly $(p+1)-\frac{p+3}{2}=\frac{p-1}{2}$ projection polynomials that are constant. The remaining projection polynomials are of type \ref{thm:main} \ref{item:a}

The purpose of this section is to show that for large enough primes this result  follows from Theorem \ref{thm:main}. 
In order to do so we have to introduce the basic notion and some technique of the discrete Fourier transformation. We will also reformulate it in terms of the projection polynomials.

A character of a finite abelian group is a homomorphism to the multiplicative group of the complex numbers. 
For a function $f \colon 
\F_p^2 \to \mathbb{C}$ and a character $\xi$ let $\hat{f}(\xi)=\sum_{x \in \F_p^2} f(x) 
\overline{\xi(x)} $. Then $\hat{f}$ is called the Fourier transform of $f$. The set of characters of $\F_p^2$ will be denoted by $\widehat{\F_p^2}$.  Any character of $\F_p^2$ 
can be written as $\xi_v(x)=e^{\frac{2 \pi i}{p}\langle x,v \rangle}$, where $\langle x,v \rangle$ denotes the standard $\pmod p$ scalar product of the vectors $v, x \in \F_p^2$ (where $\F_p^2$ and $\widehat{\F_p^2}$ are identified). 
Thus the set of characters can be endowed with a group structure that is isomorphic to the original group $\F_p^2$. This identification also assigns to each nontrivial character a direction in $\F_p^2$.

The Plancherel formula  says that 
\begin{equation}\label{eq:Plancherel}
\sum_{x \in \F_p^2} |f(x)|^2=\frac{1}{p^2}\sum_{\xi \in \widehat{\F_p^2}} |\hat{f}(\xi)|^2.\end{equation}
In particular, for the characteristic function $\mathbf 1_S$ of the set $S$ we have \begin{equation*}\label{eq:Plancherel1}
\sum_{x \in \F_p^2} |\mathbf 1_S(x)|^2=\frac{1}{p^2}\sum_{\xi \in \widehat{\F_p^2}} |\widehat{\mathbf 1_S}(\xi)|^2.\end{equation*}
We will apply this for a set $S$ of cardinality $p$ determining exactly $\frac{p+3}{2}$ directions, so we have $\sum |\mathbf 1_S(x)|^2=p$. 

It is easy to see that if $\xi$ is not the trivial character, then 
\begin{equation}\label{eq:projpol}
   \hat{\mathbf 1}_S(\xi)=\sum_{i=0}^{p-1} m(i) \sigma^i,
\end{equation} 
where $m$ denotes an affine transformation of the projection polynomial corresponding to $\xi$ and $\sigma$ is a primitive $p$'th root of unity. 

Let $m$ be a projection polynomial of a set of cardinality $p$.
It follows from Proposition \ref{prop:+2} that $m$ is either the constant $1$ polynomial or of degree at most $\frac{p-1}{2}=\frac{p+3}{2}-2$. On the other hand \cite[Theorem 1.1]{somlai} implies that if $m$ is not constant, then $m$ is of degree at least $\frac{p-1}{2}$ so we obtain that $m$ is either constant or of degree $\frac{p-1}{2}$.
The polynomials of degree $\frac{p-1}{2}$ whose range sum is $p$ are described in Theorem \ref{thm:main}.

\begin{lemma}\label{lem:fourier} \ Let $m$ denote an affine transformation of the projection polynomial corresponding to $\xi$. \vspace{-0.3cm}
\begin{enumerate}
    \item 
If $m$ is constant, then  $\hat{\mathbf 1}_S(\xi)=0$. 
    \item 
If the leading coefficient of $m$ is $\pm 1$, then $|\hat{\mathbf 1}_S(\xi)|=\sqrt{p}$. 
    \item 
If the leading coefficient of $m$ is $\pm \frac{p-1}{2}$, then $\frac{p}{2}- \frac{\sqrt{p}}{2} \le  |\hat{\mathbf 1}_S(\xi)|\le \frac{p}{2}+ \frac{\sqrt{p}}{2}$. 
\end{enumerate}    
\end{lemma}
\begin{proof}
    \begin{enumerate}
        \item Follows trivially from the fact that the sum of the $n$'th roots of unity is $0$ for every positive integer $n$.
        \item 
It is well-known \cite{Gsum} that the Gauss sum $\sum_{x \in \F_p} e^{\frac{2 \pi i}{p}x^2}$ is of absolute value $\sqrt{p}$.         

As we have seen in equation \eqref{eq:projpol}, $\hat{\mathbf 1}_S(\xi)=\sum m(i) \sigma^i$. If the leading coefficient of $m$ is $\pm 1$, then $m(x)$ can be obtained from $x^{\frac{p-1}{2}}+1$ by a change of variable $x \to ax+b$.

Now $m(x)=x^{\frac{p-1}{2}}+1$ is 0 at the quadratic non-residues, 2 at the quadratic residues and 1 at 0. This means that 
$$\hat{\mathbf 1}_S(\xi)= \sum_{i=0}^{p-1} m(i) \sigma^i =   \sum_{x \in \F_p} e^{\frac{2 \pi i}{p}x^2}$$ so $$|\hat{\mathbf 1}_S(\xi)|=\sqrt{p}.$$

Now $\sum m(ai+b) \sigma^i$ can also be expressed as 
$\sum m(i) \sigma^{a'i+b'}$ for some $a' \in \F_p^*, b' \in \F_p$. The effect by the translation by $b'$ is just a multiplication with a number of absolute value $1$. Further $\sigma^x \to \sigma^{a'x}$ can be obtained using a Galois automorphism in $\mathrm{Gal}(\mathbb{Q}(\sigma)|\mathbb{Q})$. For this particular quadratic Gauss sum, the resulting Galois conjugate differs from the original sum only by the factor $\left(\frac{a'}{p}\right)=\pm 1$, after the change of variable $y=a'x$; hence it has the same absolute value, finishing the proof of this case.


        \item 
        As we have seen in the previous case we only have to verify our statement when $m(x)=\frac{p+1}{2}\left(x^{\frac{p-1}{2}}+1 \right)$.
  Now treating the values of $m(x)$ as nonnegative rational numbers in $\{0,1, \ldots, p-1\}$ we obtain that
\[
m(x)=\frac{p+1}{2}\left(x^{\frac{p-1}{2}}+1 \right)= \frac{p}{2} \mathbf{1}_0 + \frac{1}{2} \left( \legendre{x}{p}+1 \right). 
\]

Therefore, for every nontrivial $\sigma$, we have
\[
\sum_{x\in\F_p}m(x)\sigma^x
=
\frac p2+\frac12\sum_{x\in\F_p}\left(\frac{x}{p}\right)\sigma^x,
\]
because $\sum_{x\in\F_p}\sigma^x=0$. The second term is half of a
quadratic Gauss sum, hence its absolute value is $\sqrt p/2$. This gives
\[
\frac p2-\frac{\sqrt p}{2}
\le
\left|\sum_{x\in\F_p}m(x)\sigma^x\right|
\le
\frac p2+\frac{\sqrt p}{2}.
\]
Affine changes of variable are handled as in the previous case.

\end{enumerate}
\end{proof}
Using Lemma \ref{lem:fourier} we are able to estimate 
the right hand side of equation \eqref{eq:Plancherel} for $f=\mathbf 1_S$.

If $\xi$ is the trivial character, then $\hat{f}(\xi)=p$ so $|\hat{f}(\xi)|^2=p^2$.

If the polynomial $m$ in equation \eqref{eq:projpol} is constant, then $\hat{f}(\xi)=0$. Since $|S|=p$ we have that in this case, each line orthogonal to $\xi$ (the representations of $\F_p^2$ are identified with elements of $\F_p^2$ so they are considered vectors in this vector space) contains exactly $1$ element of $S$ so $S$ does not determine the direction orthogonal to $\xi$.
We have exactly $(p+1)-\frac{p+3}{2}=\frac{p-1}{2}$ directions not determined by $S$ so we have exactly 
$(p-1)\frac{p-1}{2}$ roots of $\hat{\mathbf 1}_S$, the representations corresponding to these directions.

Let $M$ denote the number of directions such that the corresponding projection polynomial $m$ is of leading coefficient $\pm 1$ as in Lemma \ref{lem:fourier}, case 2. 
Note that the projection polynomial towards a direction is only determined up to an element of $\mathrm{AGL}(1,p)$ but these transformations can only change the leading coefficient of $m$ up to a $\pm 1$ factor. 

Then it does follow from equation \eqref{eq:Plancherel} and by Lemma \ref{lem:fourier} that
\begin{equation}\label{eq:fourier 2 irany}
\begin{split}
&\frac{1}{p^2} \left( p^2+ ( p-1 )
\left[
\frac{p-1}{2} 0 + M \sqrt{p}^2 + \left( \frac{p+3}{2}-M \right) \left( \frac{p}{2}- \frac{\sqrt{p}}{2} \right)^2
\right]
\right) \le \\ 
 &p=\frac{1}{p^2}\sum |\widehat{\mathbf 1}_S(\xi)|^2 \le \frac{1}{p^2} \left(p^2 + (p-1)
\left[
\frac{p-1}{2} 0 + M \sqrt{p}^2 + \left( \frac{p+3}{2}-M \right) \left( \frac{p}{2}+ \frac{\sqrt{p}}{2} \right)^2
\right]
\right).    
\end{split}
\end{equation}

From the lower bound in equation \ref{eq:fourier 2 irany} we immediately obtain that $\frac{p+3}{2}-M\le 5$ and case by case calculation shows that $\frac{p+3}{2}-M\le 2$ if $p \ge 47$. If $\frac{p+3}{2}-M=0$, then $p^2 +(p-1)\frac{p+3}{2}p=p^3$, which holds only for $p=3$. If $\frac{p+3}{2}-M=1$, then the reduced equation satisfies that \[ \left(\frac{p}{2}-\frac{\sqrt{p}}{2}\right)^2\le \frac{(p-1)p}{2}\le \left(\frac{p}{2}+\frac{\sqrt{p}}{2}\right)^2. \] 
Here the lower bound holds if $p>2$, while the upper bound is equivalent to \[ p-3\le 2\sqrt p, \] which implies $p\le 9$. Hence the case $\frac{p+3}{2}-M=1$ is impossible under the assumption $p\ge 47$. Thus \[ \frac{p+3}{2}-M=2 \] for $p\ge 47$.


It remains to translate this calculation on the Fourier coefficients to a combinatorial result. We have obtained that there are two directions when the corresponding projection polynomials $m_1$ and $m_2$ are of degree $\frac{p-1}{2}$ with leading coefficient $\pm \frac{p-1}{2}$. 
Then it follows from Theorem \ref{thm:main}
that there is an $y_i \in \F_p$ for $i=1,2$ such that $m_i(y_i)=\frac{p+1}{2}$  (and $m_i(y_i+z)=0 \mbox{ or } 1$  
depending only on whether $z\neq 0$ is a quadratic residue or not). 

It does simply mean that there are two non-parallel lines,  each of them containing $\frac{p+1}{2}$ elements of $S$. 
Then the intersection of these lines must belong to $S$ since $|S|=p$, moreover there can not be further points of $S$ outside the two lines in view. 
Note that the nonzero elements of the support of $\frac{p+1}{2}(x^\dd+1)$ 
and  
$\frac{p+1}{2}(-x^\dd+1)$ are the set of quadratic residues and the set of quadratic nonresidues, respectively. This simply means that we have two parallel classes of lines and on both of these classes, $p+1$ points lie on the lines indexed by $0$. The remaining points are found on the lines indexed by quadratic residues or quadratic nonresidues. 

Thus, after applying a suitable identification of these two lines with  
 $\F_p$, we may assume that the point in the intersection of the lines coincide with $0$ on both lines and the remaining points of $S$ can be identified with the quadratic residues or quadratic nonresidues, just as it is described in Theorem \ref{thm:lovaszschrijver}.

Note that the previous argument shows that if  $p \ge 47$ holds for the prime $p$ and the statement of Theorem \ref{thm:main} holds, then for the same prime $p$ we can derive the uniqueness result of Lovász and Schrijver.\qed

Further, note that Sz\H{onyi} \cite{szonyi2lines} classified the sets contained in the union of two lines, with minimal possible number of directions. This result could also have been applied to finish the proof of Theorem \ref{thm:lovaszschrijver}.

\section{Concluding remarks and open problems}
In this paper we characterised those polynomials $P$ over $\ff_p$ for which 
$\sum_{x \in \F_p} [P(x)]=p$, and $\deg(P)$ is minimal; provided that $p$ is large enough. We strongly believe that the latter condition is only technical, and one can get rid of it.
Several natural questions arise here. The first is concerning stability.
\begin{problem}
    Suppose that $\sum_{x \in \F_p} [P(x)]=p$ for a polynomial $P$ over $\ff_p$, but $P$ is not an affine translate of either $x^{\frac{p-1}{2}}+1$
    or $\frac{p+1}{2}\left( x^{\frac{p-1}{2}}+1 \right)$. Does this imply $\deg(P)\geq C\cdot (p-1)$ for some constant $C>\frac{1}{2}$, independent of $p$?
\end{problem}
We conjecture that for $p$ large enough, this holds even for $C=\frac{2}{3}$. This would imply the result of Gács concerning the number of determined directions \cite{gacs}.

The next question is about further results on small range sums.

\begin{problem}\label{probl:2} Fix a positive integer $k>1$.
    Suppose that $\sum_{x \in \F_p} [P(x)]=k\cdot p$ for a polynomial $P$ over $\ff_p$. Is it true that apart from finitely many primes $p$,  either $P$ is an affine translate of $c\cdot (x^{\frac{p-1}{2}}+1)+d$ for some constants $c,d$, or   $\deg(P)> \frac{p-1}{2}$ holds? 
\end{problem}
\vspace{-0.3cm}
{Notice that the range sum of the polynomials $(x-1)(x-2)+1, (x-1)(x-2)+2 \in \mathbb{F}_7$ is $14$. Moreover, if $p=5$, then the non-constant polynomials of degree less than $\frac{p-1}{2}$ are linear. These are permutation polynomials so the range sum is $10$ in this case. These examples show that we may only expect an affirmative answer for the question raised in Problem \ref{probl:2} when $p$ is large enough.} 

Finally, we are interested in general in the relation between $h(p)=\sum_{x \in \F_p} [P(x)]$ and $$\min\deg \{ P: \sum_{x \in \F_p} [P(x)]=h(p) \mbox{ and } P \ne 1 \}.$$

\end{document}